\documentclass[11pt,twoside]{amsart}
\usepackage{amsfonts}
\usepackage{amsmath}
\usepackage{amssymb}
\usepackage[french]{babel}
\usepackage[latin1]{inputenc}
\usepackage{graphicx}
\usepackage{mathrsfs}
\usepackage{pst-osci}%pour oscilloscope
\input xy
\xyoption{all}

\newtheorem{theo}{Th\'eor\`eme}
\newtheorem{cor}{Corollaire}
\newtheorem{lem}{Lemme}
\newtheorem{prop}{Proposition}
\theoremstyle{definition}
\newtheorem{defn}{D\'efinition}

\theoremstyle{remark}
\newtheorem{rem}{\bf Remarque\/}
\newtheorem{rems}[rem]{\bf Remarques\/}

\numberwithin{equation}{section}

\def\C{{\mathbb{C}}}
\def\P{{\mathbb{P}}}
\def\N{{\Bbb N}}

\def\1{{\mathchoice {\rm 1\mskip-4mu l} {\rm 1\mskip-4mu l}
{\rm 1\mskip-4.5mu l} {\rm 1\mskip-5mu l}}}            %fonction indicatrice d'un ensemble

\newcommand{\ds}       {\displaystyle}
\newcommand{\w}{\wedge}
\newcommand{\rt}{\longrightarrow}
\newcommand{\mo}{\smallsetminus}
\title{Ordres des courants positifs pluriharmoniques}

\author[K. Dabbek]{khalifa Dabbek}
\email{khalifa.dabbek@fsg.rnu.tn}
\author[N. Ghiloufi]{Noureddine Ghiloufi}
\email{noureddine.ghiloufi@fsg.rnu.tn}
\address{D\'epartement de Math\'ematiques \\ Facult\'e des sciences de Gab\`es \\ Universit\'e de Gab\`es \\ 6033 Gab\`es Tunisie.}
   % ams classification ---> amsart.cls page 794


\begin{document}
\maketitle
\hrule
\vskip.5cm
\begin{abstract}
    Dans cet article, nous \'etudions l'ordre (d'alg\'ebricit\'e) d'un courant positif pluriharmonique et nous le comparons soit avec l'ordre de ses tranches concourantes soit avec ses ordres directionnels. Des estimations de croissance de la fonction de \textsc{Lelong} sont \'etablies dont le probl\`eme d'alg\'ebricit\'e du courant est trait\'e comme cons\'equence.\\

    \textbf{Orders of positive pluriharmonic currents}\\
    \textsc{Abstract.} In this article, we study the order of a positive pluriharmonic current and we compare it with either the order of the concurrent slices or the directionnel orders of the current. Therefore some estimates of the growth of the \textsc{Lelong} function are established and the problem of algebraicity of the current is treated as a result.
\end{abstract}
\vskip.5cm
\hrule

%\tableofcontents
\section{Pr\'eliminaires}
    Dans tout ce travail on utilise les notations suivantes: Pour $r>0$, $\mathbb{B}(r)=\mathbb{B}_n(r):=\{z\in\C^n;\ |z|<r\}$  la boule euclidienne de $\C^n$ de centre 0 et de rayon $r$ et pour tous $0<r_1<r_2$, $\mathbb{B}(r_1,r_2):=\{z\in\C^n;\ r_1\leq |z|<r_2\}=\mathbb{B}(r_2)\mo\mathbb{B}(r_1)$ ainsi que les op\'erateurs $$\partial=\sum_{j=1}^n{\partial \over\partial z_j}\ dz_j, \quad \overline{\partial} =\sum_{j=1}^n{\partial \over\partial \overline{z}_j}\ d\overline{z}_j,\quad d=\partial+\overline{\partial}\quad\hbox{et}\ d^c={i\over 4\pi}(\overline{\partial}-\partial).$$
    Notons ${\mathscr D}_{p,q}(\C^n)$ l'espace des formes diff\'erentielles de classe ${\mathscr C}^\infty$ \`a supports compacts de bidegr\'e $(p,q)$ dans $\C^n$. L'espace des courants de bidimension $(p,q)$ (ou de bidegr\'e $(n-p,n-q)$), not\'e ${\mathscr D}_{p,q}'(\C^n)$, est par d\'efinition le dual de ${\mathscr D}_{p,q}(\Omega)$ muni de sa topologie usuelle.\\
    Soit $T\in {\mathscr D}'_{p,p}(\C^n)$; on dit que $T$ est positif si $T\w i\alpha_1\w\overline{\alpha}_1 \w ...\w i\alpha_p\w\overline{\alpha}_p$ est une mesure positive pour toutes $\alpha_1,...,\alpha_p\in {\mathscr D}_{1,0}(\C^n)$. Le courant $T$ est dit ferm\'e si $\langle dT,\ \phi\rangle:=-\langle T,\ d\phi\rangle=0$ pour tout $\phi\in {\mathscr D}_{p-1,p}(\C^n)$, il est dit plurisousharmonique (psh)(resp. pluriharmonique (ph)) si $dd^cT$ est un courant positif (resp. $dd^cT=0$) o\`u $\langle dd^cT,\ \phi\rangle:=\langle T,\ dd^c\phi\rangle$ pour tout $\phi\in {\mathscr D}_{p-1,p-1}(\C^n)$.\\
    On associe \`a un courant positif  $T$ de bidimension $(p,p)$ sur $\C^n$, la fonction de \textsc{Lelong} d\'efinie par $\nu_T(r)=\frac1{r^{2p}}\int_{\mathbb{B}(r)}T\w(dd^c |z|^2)^p$. Si $T$ est positif plurisousharmonique alors $\nu_T$ est croissante sur $]0,+\infty[$. Un courant positif $T$ est dit alg\'ebrique si la fonction $\nu_T$ est born\'ee (Pour $T=[X]$, le courant d'int\'egration sur un ensemble analytique $X$, l'alg\'ebricit\'e de $T$ est \'equivalente \`a l'alg\'ebricit\'e classique de $X$).\\

    Soient $k\leq p<n$ et la projection canonique $\pi:\C^n\rt \C^k$ d\'efinie par $\pi(z',z'')=z'$.
    Soit $h$ une fonction positive bor\'elienne born\'ee \`a support compact dans la boule unit\'e de $\C^k$ telle que $\int_{\C^k}h(z')d\lambda_k(z')=1$. Pour $\epsilon>0$, on pose $h_\epsilon(z')=\epsilon^{-2k}h(z'/\epsilon)$.  Soient $T\in{\mathscr D}'_{p,p}(\C^n)$ un courant positif et $a\in \C^k$, la tranche (parall\`ele) de $T$ par $h$ au point $a$ not\'e $\langle T,\pi,a\rangle_h$, est la limite faible dans ${\mathscr D}'_{p-k,p-k}(\C^n)$, quand elle existe, de $T\w\pi^*(h_\epsilon(z'-a).(dd^c|z'|^2)^k)$ quand $\epsilon\rt 0$. Si $T$ est un courant positif pluriharmonique, D'apr\`es \cite{Da-Elk-El}, il existe un ensemble de mesure de \textsc{Lebesgue} nulle  de $\C^k$ en dehors duquel la tranche de $T$ existe et est ind\'ependante de $h$.\\

    Dans la suite soient $p,\ q<n$ des entiers tels que $p+q\geq n$ et $G_{q,n}$ la Grassmannienne des sous-espaces vectoriels de dimension $q$ dans $\C^n$ munie de sa forme K\"ahlerienne standard $\omega_q$ et  la m\'etrique de \textsc{Fubini-Study} associ\'ee not\'ee $\mu_q$ (ou simplement $\mu$ s'il n'y a pas d'ambigu\"{\i}t\'e). Soit $X_{q,n}=\{(z,L)\in\C^n\times G_{q,n};\ z\in L\}$ le fibr\'e vectoriel de rang $q$ au dessus de $G_{q,n}$, muni des projections canoniques $\pi:\ X_{q,n}\rt G_{q,n}$ et $\sigma:\ X_{q,n}\rt \C^n$; alors $\beta_{X_{q,n}}:=\pi^*\omega_q+\sigma^*\beta$ d\'efinie une forme K\"ahlerienne sur la vari\'et\'e $X_{q,n}$ qui est de dimension $dim\: X_{q,n}=n+(q-1)(n-q)$. La restriction $\sigma_0$ de $\sigma$ \`a $X_{q,n}':=\sigma^{-1}(\C^n\mo\{0\})$ est une submersion sur $\C^n\mo\{0\}$. Soit $T$ un courant positif pluriharmonique sur $\C^n$, alors $\sigma_0^*T$ d\'efinit un courant positif pluriharmonique sur $X_{q,n}'$ de masse localement finie au voisinage de $\sigma^{-1}(0)$ (cf. \cite{BM-El}), et d'apr\`es \textsc{Dabbek-Elkhadhra-El Mir} \cite{Da-Elk-El}, l'extension trivial de $\sigma_0^*T$ par z\'ero au dessus de $\sigma^{-1}(0)$ est un courant positif de $dd^c-$n\'egatif de dimension $p+(q-1)(n-q)$ sur $X_{q,n}$, qu'on notera $\widetilde{\sigma_0^*T}$. On pose alors\\
        $\bullet\ \sigma^*T=\widetilde{\sigma_0^*T}$ si $q>n-p$, o\`u $n-p$ est le degr\'e de $\widetilde{\sigma_0^*T}$.\\
        $\bullet\ \sigma^*T=\widetilde{\sigma_0^*T}+\nu(T,0)[\sigma^{-1}(0)]$ si $q=n-p$, o\`u $\nu(T,0)$ est le nombre de \textsc{Lelong} de $T$ en 0.\\
    Comme pr\'ec\'edemment, il existe un ensemble $E_T$ de la grassmannienne $G_{q,n}$ de mesure de \textsc{Fubini-Study} nulle tel que la tranche (concourante) $\langle\ \sigma^*T,\pi,L\ \rangle $ de $T$, qu'on notera ici $T_{|L}$, existe pour tout $L\in \Omega_T:=G_{q,n}\mo E_T$. Quand $T$ est \`a coefficients continus, cette tranche coincide avec la restriction usuelle de $T$ \`a $L$. De plus en dehors d'un ensemble n\'egligeable de $\Omega_T$, $T_{|L}$ est un courant pluriharmonique; on peut donc supposer dans toute la suite que $T_{|L}$ est bien d\'efini et est pluriharmonique sur $\Omega_T$.\\
    On note par $$\nu_{T_{|L}}(r):=\frac{1}{r^{2(p+q-n)}}\int_{X_{q,n}(r)}T_{|L}\w\sigma^*(dd^c|z|^2)^{p+q-n}$$ la fonction de \textsc{Lelong} de $T_{|L}$ avec $X_{q,n}(r)=\sigma^{-1}(\mathbb{B}(r))$.

    \begin{defn}
        Un courant positif est dit \textit{d'ordre d'alg\'ebricit\'e} $\varrho$ fini si la limite suivante est finie
        $$\varrho:=\limsup_{r\to +\infty}\frac{\log\nu_T(r)}{\log r}<+\infty.$$
    \end{defn}
    Dans toute la suite  on utilise seulement ordre pour indiquer l'ordre d'alg\'e\-bricit\'e. A noter que l'ordre du courant d'int\'egration sur un ensemble analytique $X$, est \'egale \`a l'ordre classique de l'ensemble analytique $X$ (cf. \cite{Ch}).
\section{Ordres des tranches concourantes}
\subsection{Th\'eor\`eme principal et cons\'equences}
    Le r\'esultat principal de cette partie (th\'eor\`eme \ref{th1}) consiste \`a contr\^oler la fonction $\nu_T$ du courant $T$ par celles de ses tranches sur la grassmannienne.

    \begin{theo}\label{th1}
        Soient $T$ un courant positif pluriharmonique  de bidimension $(p,p)$ sur $\C^n$ et $E$ un ensemble borelien de $\Omega_T$ de mesure de \textsc{Fubini-Study} $\mu(E)$ non nulle. Alors il existe deux constantes $c_1,\ c_2>0$ qui d\'ependent de $E$ telles que l'on ait
        $$c_1 \nu_T(c_2 r)\leq \int_E \nu_{T_{|L}}(r)d\mu(L)$$
        pour tout $r$ suffisamment grand.
    \end{theo}
    \begin{cor}\label{cor1}
        Soit $T$ un courant positif pluriharmonique de bidimension $(p,p)$ sur $\C^n$. On suppose qu'il existe un ensemble $E$ de mesure  non nulle de la grassmannienne $G_{q,n}$ tel que $T_{|L}$ soit un courant alg\'ebrique pour tout $L\in E$, alors $T$ est alg\'ebrique.
    \end{cor}
    Notons que ces deux r\'esultats sont d\'emontr\'es par les m\^emes auteurs dans \cite{Gh-Da} dans le cas d'un courant positif ferm\'e.\\

    \begin{proof}
        Pour tout $N\in\N^*$, on pose $\mathscr B_N:=\{L\in E;\ \nu_{T_{|L}}(r)\leq N,\ \forall\;r>0\}$. Comme $\mu(E)>0$  et, par hypoth\`ese, $\bigcup_{N}\mathscr B_N=E$ donc il existe $N_0>0$ tel que $\mu(\mathscr B_{N_0})>0$. D'apr\`es le th\'eor\`eme \ref{th1}, il existe deux constantes $c_1,\ c_2>0$ telles que
        $$\nu_T(r)\leq \frac1{c_1}\int_{\mathscr B_{N_0}}\nu_{T_{|L}}\left(\frac r{c_2} \right)d\mu(L)\leq \frac1{c_1}N_0\mu(\mathscr B_{N_0}).$$
    \end{proof}

    \begin{cor}\label{cor2}
        Soit $T$ un courant positif pluriharmonique de bidimension $(p,p)$ sur $\C^n$. Soit $(r_m)_m$ une suite  de r\'eels positifs croissante vers $+\infty$. Alors l'ensemble
        $$ E:=\left\{L\in \Omega_T;\ \lim_{m\to +\infty}\frac{\nu_{T_{|L}}(\alpha r_m)}{\nu_T(r_m)}=0,\ \hbox{pour tout }\alpha>0\right\}$$ est de mesure nulle.
    \end{cor}
    Ce corollaire est d\'emontr\'e par \textsc{Amamou-Ben Farah} \cite{Am-BF} dans le cas des courants positifs ferm\'es et $q=1$, ce qui correspond \`a l'espace projectif $\mathbb{P}^{n-1}$.

    \begin{proof}
        Supposons que $\mu(E)>0$. Pour $s\in\N$, on consid\`ere
        $$ E_s:=\left\{L\in \Omega_T;\ \lim_{m\to +\infty}\frac{\nu_{T_{|L}}(s r_m)}{\nu_T(r_m)}=0\right\}.$$ On a $E=\cup_s E_s$ et d'apr\`es le th\'eor\`eme d'\textsc{Egorov}, pour tout $s\in\N$ il existe $K_s\subset E_s$ de mesure $\mu(K_s)\leq \frac{\mu(E)}{2^{s+2}}$ tel que la suite $\left(\frac{\nu_{T_{|L}}(s r_m)}{\nu_T(r_m)}\right)_m$ converge uniform\'ement vers 0 sur $E_s\mo K_s$. Si on note par  $W=E\mo \cup_sK_s$, alors $\mu(W)\geq \mu(E)/2$ et la suite $\left(\frac{\nu_{T_{|L}}(\alpha r_m)}{\nu_T(r_m)}\right)_m$ converge uniform\'ement vers 0 sur $W$ ($\forall\;\alpha$) ce qui donne
        $$\lim_{m\to +\infty} \int_W \frac{\nu_{T_{|L}}(\alpha r_m)}{\nu_T(r_m)}d\mu(L)=0,\quad \forall\; \alpha>0.$$ Comme $\mu(W)>0$, le th\'eor\`eme \ref{th1} implique l'existence de deux constantes $c_1,\ c_2>0$ tels que
        $$ c_1\leq \int_W \frac{\nu_{T_{|L}}(\frac1{c_2} r_m)}{ \nu_T(r_m)}d\mu(L)$$ pour $m$ suffisamment grand  ce qui est en contradiction avec la limite est nulle quand $m$ tend vers l'infinie.
    \end{proof}
    \begin{cor}
        Soit $T$ un courant positif pluriharmonique de bidimension $(p,p)$ sur $\C^n$. S'il existe un ensemble $E\subset\Omega_T$ de mesure non nulle de la grassmannienne $G_{q,n}$ tel que $T_{|L}$ soit nul pour tout $L\in E$ alors $T$ est nul.
    \end{cor}
    \begin{proof}
        Comme $\mu(E)>0$, d'apr\`es le th\'eor\`eme \ref{th1}, il existe $c_1,\ c_2>0$ tels que pour tout $r$ suffisamment grand on a
        $$\nu_T(r)\leq \frac1{c_1}\int_E \nu_{T_{|L}}\left(\frac r{c_2}\right)d\mu(L)=0.$$
        Ce qui donne $T=0$.
    \end{proof}
\subsection{Preuve du th\'eor\`eme \ref{th1}}
    la d\'emonstration du th\'eor\`eme \ref{th1} se fait par r\'ecurrence sur l'entier $q$ o\`u on utilise les lemmes \ref{lem1} et \ref{lem2} qui suivent.
    \begin{lem}\label{lem1}(formule de type \textsc{Crofton})
        Soit $S$ un courant positif pluriharmonique de bidimension $(p,p)$ sur $\C^n$. Alors pour tout $r>0$ on a $$\nu_S(r)=\int_{G_{q,n}}\nu_{S_{|L}}(r)d\mu(L).$$
    \end{lem}
    \begin{proof}
        Soit $\chi_k$ un noyau r\'egularisant qui ne d\'epend que de $|z|$ sur $\C^n$, on note par $S_k=S*\chi_k$ le r\'egularis\'e de $S$ qui est un courant positif pluriharmonique de classe $\mathcal C^\infty$ sur $\C^n$. D'apr\`es \textsc{Alessandrini-Bassanelli} \cite{Al-Ba}, la suite $(\sigma^*(S_k))_k$ est born\'ee en masse, alors quitte \`a extraire une sous suite, on peut supposer que $(\sigma^*S_k)_k$ converge faiblement sur $X_{q,n}$ vers $\sigma^*S$. D'apr\`es la formule de \textsc{Lelong-Jensen}, pour $0<r_1<r_2<r$,
        $$\frac{1}{r_2^{2p}}\int_{\mathbb{B}(r_2)}S_k\w (dd^c|z|^2)^p-\frac{1}{r_1^{2p}}\int_{\mathbb{B}(r_1)}S_k\w (dd^c|z|^2)^p=\int_{\mathbb{B}(r_1,r_2)}S_k\w (dd^c\log|z|^2)^p.$$
        Donc si $r_1\to 0^+$, $\ds\frac{1}{r_2^{2p}}\int_{\mathbb{B}(r_2)}S_k\w(dd^c|z|^2)^p = \int_{\mathbb{B}(r_2)\mo \{0\}}S_k\w (dd^c\log|z|^2)^p$.\\
        Quitte \`a remplacer $S$ par $S\w(dd^c|z|^2)^{p-q}$, on peut supposer que $p+q=n$, donc on peut appliquer l'\'egalit\'e (prouv\'e par \textsc{Siu} (\cite{Siu}, p128)):
        $$(dd^c\log|z|^2)^p=\sigma_*\pi^*\omega_{n-p}^{p(n-p)}$$ o\`u $\omega_{n-p}$ est la forme K\"ahlerienne canonique de $G_{n-p,n}$, et par suite $$\begin{array}{lll}
            \ds\frac{1}{r_2^{2p}}\int_{\mathbb{B}(r_2)}S_k\w (dd^c|z|^2)^p &=&\ds\int_{\mathbb{B}(r_2)\mo\{0\}}S_k\w (dd^c\log|z|^2)^p\\
            &=&\ds\int_{\mathbb{B}(r_2)\mo\{0\}}S_k\w\sigma_*\pi^*\omega_{n-p}^{p(n-p)}\\
            & = &\ds\int_{\sigma^{-1}(\mathbb{B}(r_2))}\sigma^*S_k\w\pi^*\omega_{n-p}^{p(n-p)}.
          \end{array}$$
        Or
        $$\begin{array}{lll}
            \ds\frac{1}{r_2^{2p}}\int_{\mathbb{B}(r_1)}S\w (dd^c|z|^2)^p&\leq&\ds\liminf_{k\to+\infty} \frac{1}{ r_2^{2p}}\int_{\mathbb{B}(r_2)}S_k\w (dd^c|z|^2)^p\\
            &\leq&\ds\limsup_{k\to+\infty}\int_{\sigma^{-1}(\mathbb{B}(r_2))}\sigma^*S_k\w\pi^*\omega_{n-p}^{p(n-p)}\\
            & \leq&\ds\int_{\sigma^{-1}(\mathbb{B}(r))}\sigma^*S\w\pi^*\omega_{n-p}^{p(n-p)}
        \end{array}$$
        et
        $$\begin{array}{lll}
            \ds\int_{\sigma^{-1}(\mathbb{B}(r_1))}\sigma^*S\w\pi^*\omega_{n-p}^{p(n-p)}&\leq&\ds\liminf_{k\to+\infty}\int_{\sigma^{-1}(\mathbb{B}(r_2))} \sigma^*S_k\w\pi^*\omega_{n-p}^{p(n-p)}\\
            &\leq&\ds\limsup_{k\to+\infty}\frac{1}{ r_2^{2p}}\int_{\mathbb{B}(r_2)}S_k\w(dd^c|z|^2)^p\\
            & \leq&\ds\frac{1}{r_2^{2p}}\int_{\mathbb{B}(r)}S\w (dd^c|z|^2)^p.
        \end{array}$$
        Si on tend $r_1\to r$ ($r$ en dehors d'un ensemble au plus d\'enombrable), on obtient
        $$\frac{1}{ r^{2p}}\int_{\mathbb{B}(r)}S\w(dd^c|z|^2)^p=\int_{\sigma^{-1}(\mathbb{B}(r))}\sigma^*S\w\pi^*\omega_{n-p}^{p(n-p)}$$
        Et d'apr\`es la formule de tranchage on a
        $$\int_{\sigma^{-1}(\mathbb{B}(r))}\sigma^*S\w\pi^*\omega_{n-p}^{p(n-p)}=\int_{L\in G_{n-p,n}}\left(\int_{L\cap \mathbb{B}(r)}S_{|L}\right)\omega_{n-p}^{p(n-p)}.$$ D'o\`u $\ds\frac{1}{r^{2p}}\int_{\mathbb{B}(r)}S\w (dd^c|z|^2)^p=\int_{L\in G_{n-p,n}}\left(\int_{L\cap \mathbb{B}(r)}S_{|L}\right)\omega_{n-p}^{p(n-p)}.$
    \end{proof}

%%%%%%%%%%%%%%%%%%%%%%%%%%%%%%%%%%%%%%%%%%%%%%%%%%%%%%%

    \begin{rem}
        Dans le cas des courants positifs ferm\'es, la formule de \textsc{Crofton} est d\'emontr\'e par \textsc{Siu} \cite{Siu} en 1974. Une question naturelle se pose: A-t-on une formule pareille pour les courants positifs plurisousharmoniques?
    \end{rem}

    \begin{lem}\label{lem2} \cite{Da-Elk-El}
        Soit $S$ un courant positif pluriharmonique de bidimension $(p,p)$  sur un ouvert $\mathcal O$ de $\C^n$. Soit $f$ une fonction psh , $f\geq -1$ de classe $\mathcal C^2$ sur $\mathcal O$ telle que $\mathcal O'=\{z\in \mathcal O;\ f(z)<0\}$ soit relativement compact dans $\mathcal O$. Si $K$ un compact de $\mathcal O'$, on pose $c_K=-\sup_{z\in K}f(z)$.\\
        Alors pour tout entier $1\leq s\leq p$ et pour toute fonction $g$ psh  de classe $\mathcal C^2$ sur $\mathcal O'$ v\'erifiant $-1\leq g<0$ on a: $$\int_K S\w(dd^cg)^p\leq c_K^{-s}\int_{\mathcal O'}S\w(dd^cf)^s\w(dd^cg)^{p-s}.$$
    \end{lem}
    Si $S$ est de classe $\mathcal C^2$ alors le lemme reste vrai en omettant  l'hypoth\`ese de r\'egularit\'e de  $f$ et $g$.

    \begin{proof} \textit{(du th\'eor\`eme \ref{th1})} On proc\`ede par r\'ecurrence sur $q$.\\
    \begin{enumerate}
        \item[$\bullet$] \textbf{$q=1$:} Soit $T$ un courant positif pluriharmonique de bidimension $(n-1,n-1)$ sur $\C^n$ et $E\subset\Omega_T\subset \mathbb P^{n-1}$ de mesure non nulle. Si $w\in\mathbb P^{n-1}$, on note $L_w$ l'hyperplan de $\C^n$, d'\'equation $w_1z_1+...+w_nz_n=0$ et $[L_w]$ le courant d'int\'egration sur $L_w$. La fonction $\ds f_w(z)=\frac{w_1z_1+...+w_nz_n}{|w|}$ est psh sur $\C^n$ et $[L_w]=dd^c\log|f_w|$. Soit $v$ la fonction d\'efinie sur $\C^n$ par
            $$v(z)=\int_{\mathbb P^{n-1}}\log|f_w(z)|d\mu_E(w)\quad \hbox{o\`u}\quad\mu_E =\frac1{\mu(E)}\mu_{|E}$$
            alors $v$ est psh sur $\C^n$ qui v\'erifie $\log|z|-\eta\leq v(z)\leq \log|z|$ pour tout $z\in\C^n$ o\`u $\eta>0$ est une constante.\\

            Soit $\chi_j$ un noyau r\'egularisant qui ne  d\'epend que de $|z|$ sur $\C^n$, on note par $T_j=T*\chi_j$ le r\'egularis\'e de $T$ qui est un courant positif pluriharmonique de classe $\mathcal C^\infty$ sur $\C^n$.\\

            Fixons $r>0$ et soit $1<\delta<2$. Soit $\phi$ la fonction psh sur $\C^n$ d\'efinie par $\phi(z)=\max(-1,v(z)-\log(\delta r))$.\\
            Le lemme \ref{lem2}, appliqu\'e \`a  $\mathcal O=\C^n$ et $\mathcal O'=\{z\in\C^n;\ \phi(z)<0\}$ qui est inclus dans $\mathbb{B}(\delta re^\eta)$, donne
            $$\int_{|z|\leq r}T_j\w\beta^{n-2}\w dd^c\left(\left|\frac{z}{\delta re^\eta}\right|^2-1\right)\leq \frac1{\log\delta} \int_{|z|\leq \delta re^\eta}T_j\w\beta^{n-2}\w dd^c\phi$$
            et par suite
            $$\int_{|z|< r}T_j\w\beta^{n-1}\leq r^2\frac{\delta^2 e^{2\eta}}{\log\delta} \int_{|z|\leq \delta re^\eta}T_j\w\beta^{n-2}\w dd^cv.$$
            Soient $0<c_2<\delta e^\eta$ et $\psi$ une fonction $\mathcal C^\infty$ sur $\C^n$ v\'erifiant $\1_{\mathbb{B}(\delta e^\eta r)}\leq \psi\leq \1_{\mathbb{B}(\frac r{c_2})}$. On pose
            $$\begin{array}{ll}
                \bullet &\ds c_1=\frac{\mu(E)\log\delta}{\delta^2 e^{2\eta}} \\
                \bullet &\ds A_T:=\{r>0;\ ||T||(\partial \mathbb{B}(r))>0 \hbox{ ou } ||T||(\partial \mathbb{B}(\delta r))>0\}
              \end{array}$$
            $A_T$ est au plus d\'enombrable et pour $r\not\in A_T$, par passage \`a la limite quand $j\to+\infty$, le th\'eor\`eme de Fubini donne
            $$\begin{array}{lcl}
                \ds \int_{|z|< r}T\w\beta^{n-1}&\leq&\ds \frac{\mu(E)}{c_1}r^2 \int_{\C^n}\psi T\w\beta^{n-2}\w dd^cv\\
                & = &\ds \frac{r^2}{c_1}\int_E\int_{\C^n}\psi T\w dd^c\log|f_w|\w\beta^{n-2}d\mu(w) \\
                & \leq & \ds \frac{r^2}{c_1}\int_E\int_{|z|<\frac1{c_2}r}T\w [L_w]\w\beta^{n-2}d\mu(w)
            \end{array}$$
            Ce qui donne le r\'esultat pour le cas $q=1$.
        \item[$\bullet$]
            Supposons que le r\'esultat est vrai \`a l'ordre $q\geq1$, c'est \`a dire sur $G_{q,n}$, et prouvons le \`a l'ordre $q+1$. L'id\'ee est de ramener le probl\`eme de $G_{q+1,n}$ vers $G_{q,n}$, pour cela, on consid\`ere l'ensemble:
            $$Y_{q,n}=\{(L,\Lambda)\in G_{q,n}\times G_{q+1,n};\ L\subset \Lambda\}$$ et les projections canoniques:
            $$\xymatrix{G_{q,n}&\ar[l]_f Y_{q,n}\ar[r]^g &G_{q+1,n}\\
            X_{q,n}\ar[u]^{\pi_q}\ar[r]_{\sigma_q}&\C^n&\ar[l]^{\sigma_{q+1}}\ar[u]_{\pi_{q+1}}X_{q+1,n}}$$
            On muni $Y_{q,n}$ de la forme K\"ahlerienne $\omega=g^*\omega_q+f^*\omega_{q+1}$ et donc de la mesure de \textsc{Fubini-Study} correspondante ${\mathfrak m}$. On v\'erifie que $g_*{\mathfrak m}\otimes\mu_{q+1}=\mu_q\otimes f_*{\mathfrak m}$, de plus la valeur $\kappa={\mathfrak m}(g^{-1}(\Lambda))={\mathfrak m}(f^{-1}(L))$ est ind\'ependante de $L\in G_{q,n}$ $\mu_q-$presque partout et de $\Lambda\in G_{q+1,n}$ $\mu_{q+1}-$presque partout.\\
            On a $\sigma_q^*T$ est un courant positif de dimension $p+(q-1)(n-q)$ sur $X_{q,n}$, de m\^eme $\sigma_{q+1}^*T$ est un courant positif de dimension $p+q(n-q-1)$ sur $X_{q+1,n}$.\\
            Soit $\Lambda\in E\subset \Omega_T=G_{q+1,n}\mo E_T$.\\

            On remarque que $G_{q,q+1}(\Lambda)=f(g^{-1}(\Lambda))$, et comme $T_{|\Lambda}$ est positif pluriharmonique, d'apr\`es le lemme \ref{lem1}, on a
            $$\nu_{T_{|\Lambda}}(r)=\int_{L\in f(g^{-1}(\Lambda))} \nu_{T_{|L}}(r)d(f_*{\mathfrak m})(L)$$
            d'apr\`es Fubini on a,
            $$\begin{array}{lcl}
                \ds\int_E \nu_{T_{|\Lambda}}(r)d\mu_{q+1}(\Lambda) & = &\ds \int_E\left(\int_{L\in f(g^{-1}(\Lambda))}\nu_{T_{|L}}(r)d(f_*{\mathfrak m})(L)\right) d\mu_{q+1}(\Lambda) \\
                & = &\ds \int_{L\in f(g^{-1}(E))}\left(\int_{\Lambda\in g(f^{-1}(L))}\nu_{T_{|L}}(r)d(g_*{\mathfrak m})(\Lambda)\right)d\mu_q(L)\\
                & = &\ds \kappa  \int_{f(g^{-1}(E))}\nu_{T_{|L}}(r)d\mu_q(L).
            \end{array}$$
            Comme $f(g^{-1}(E))$ est de mesure non nulle dans $G_{q,n}$, d'apr\`es l'hypoth\`ese de r\'ecurrence, il existe $c_1,\ c_2>0$ tels que
            $$\ds c_1 \nu_T(c_2 r)\leq \int_{f(g^{-1}(E))} \nu_{T_{|L}}(r)d\mu_q(L).$$
            Donc
            $$\kappa c_1\: \nu_T(c_2 r)\leq \int_E \nu_{T_{|\Lambda}}(r) d\mu_{q+1}(\Lambda).$$
        \end{enumerate}
        \end{proof}
%%%%%%%%%%%%%%%%%%%%%%%%%%%%%%%%%%% les corollaires

%%%%%%%%%%%%%%%%%%%%%%%%% les applications
\subsection{Applications du th\'eor\`eme \ref{th1}}
    Le r\'esultat suivant est une illustration des r\'esultats pr\'ec\'edents sur les courants positifs pluriharmoniques d'ordres finis.
    \begin{theo}\label{th2}
        Soit $T$ un courant positif pluriharmonique de bidimension $(p,p)$ et d'ordre $\varrho$ fini sur $\C^n$. Alors pour presque tout $L\in G_{q,n}$, $T_{|L}$ est d'ordre $\varrho_L$ \'egal \`a $\varrho$.
    \end{theo}
    \begin{proof} La d\'emonstration se fait en deux \'etapes.\\

    \noindent\textit{Premi\`ere \'etape:  $\varrho_L\geq\varrho$, $\mu-$presque partout.}\\

        Par d\'efinition, il existe une suite $(r_m)_m$ croissante vers l'infinie telle que\\ $\log(\nu_T(r_m)).(\log(r_m))^{-1}\underset{m\to+\infty}\rt \varrho$. On consid\`ere alors l'ensemble
        $$ E=\left\{L\in \Omega_T;\ \lim_{m\to +\infty}\frac{\nu_{T_{|L}}(\alpha r_m)}{\nu_T(r_m)}=0,\ \hbox{pour tout }\alpha>0\right\}\cup \Omega_T^\complement$$ qui est de mesure nulle (d'apr\`es le corollaire \ref{cor2}). Pour $L\not\in E$, il existe $\alpha_0>0$, une sous suite, not\'ee de m\^eme $(r_m)_m$, tels que\\ $\limsup_{m\to +\infty} \nu_{T_{|L}}(\alpha_0 r_m)/ \nu_T(r_m)=\mathfrak{a}_L\in{}]0,+\infty]$. Or
        $$\frac{\log \nu_{T_{|L}}(\alpha_0 r_m)}{\log(\alpha_0r_m)}=\frac{\log \nu_{T_{|L}}(\alpha_0 r_m)-\log(\nu_T(r_m))}{\log (\alpha_0r_m)}+ \frac{\log \nu_T(r_m)}{\log r_m}\times\frac{\log r_m}{\log (\alpha_0r_m)}.$$
        Le  terme de gauche de cette \'egalit\'e  admet une limite-sup\'erieure plus petite que $\varrho_L$ quand $m$ tend vers $+\infty$, alors que le deuxi\`eme terme de droite tend vers $\varrho$.\\
        Dans le cas o\`u $\mathfrak{a}_L<+\infty$, $$\limsup_{m\to+\infty}\frac{\log \nu_{T_{|L}}(\alpha_0
        r_m)-\log(\nu_T(r_m))}{\log (\alpha_0r_m)}=\limsup_{m\to+\infty}\frac{\log\left(\frac{\nu_{T_{|L}}(\alpha_0
        r_m)}{\nu_T(r_m)}\right)}{\log (\alpha_0r_m)}=0.$$ Dans l'autre cas, $\mathfrak{a}_L=+\infty$, on a pour $m$ suffisamment grand,
        $$\frac{\log \nu_{T_{|L}}(\alpha_0 r_m)}{ \log (\alpha_0r_m)}\geq \frac{\log \nu_T(r_m)}{\log r_m}\times\frac{\log r_m}{\log (\alpha_0r_m)}.$$
        Donc $\varrho_L\geq \varrho$.\\

    \noindent\textit{Deuxi\`eme \'etape:  $\varrho_L\leq \varrho$, $\mu-$presque partout.}\\

        Pour tout $\epsilon\in{}]0,1[$ et $\gamma>1$ fix\'e, on pose
        $${\mathscr E}_{\epsilon,k}=\left\{L\in G_{q,n};\ \nu_{T_{|L}}(\gamma^k)>\left(\log \gamma^k\right)^{1+(\log(\log\gamma^k))^{\epsilon-1}} \nu_T(\gamma^k)\right\},\ k\in \N.$$ Alors
        $$\begin{array}{lcl}
            \ds\mu({\mathscr E}_{\epsilon,k})\left(\log \gamma^k\right)^{1+(\log(\log\gamma^k))^{\epsilon-1}} \nu_T(\gamma^k)& \leq &\ds \int_{{\mathscr E}_{\epsilon,k}}\nu_{T_{|L}}(\gamma^k)d\mu(L) \\
            & \leq &\ds \int_{G_{q,n}}\nu_{T_{|L}}(\gamma^k)d\mu(L)\\
            & \leq & \nu_T(\gamma^k)
        \end{array}$$
        Par suite
        $$\mu({\mathscr E}_{\epsilon,k})\leq \left(\log \gamma^k\right)^{-1-(\log(\log\gamma^k) )^{\epsilon-1}}. $$
        Posons ${\mathscr E}_\epsilon=\cap_{j\geq1}\cup_{k\geq j}{\mathscr E}_{\epsilon,k}.$ Puisque
        $$\ds\lim_{k\to+\infty}(\log k)^2\left(\log \gamma^k\right)^{-(\log(\log\gamma^k))^{\epsilon-1}}=\lim_{k\to+\infty}(\log k)^2\exp\{-(\log(\log\gamma^k))^\epsilon\}=0$$
        alors
        $$\ds\sum_{k=1}^{+\infty}\left(\log \gamma^k\right)^{-1-(\log (\log\gamma^k))^{\epsilon-1}} <+\infty$$
        et donc $\mu({\mathscr E}_\epsilon)=0$. Pour $L\not\in{\mathscr E}_\epsilon$, il existe $k_L$ tel que pour $k\geq k_L,\ \nu_{T_{|L}}(\gamma^k)\leq \left(\log \gamma^k\right)^{1+(\log(\log\gamma^k))^{\epsilon-1}}\nu_T(\gamma^k)$.\\
        Si $r$ est tel que $\gamma^{k-1}\leq r\leq \gamma^k$ alors
        \begin{equation}\label{eq 2.1}
            \nu_{T_{|L}}(r)\leq \left(\log \gamma r\right)^{1+(\log(\log r))^{\epsilon-1}}\nu_T(\gamma r)
        \end{equation}
        et ainsi
        $$\begin{array}{lcl}
            \ds \varrho_L:=\limsup_{r\rt+\infty}{\log \nu_{T_{|L}}(r)\over \log r}& \leq  &\ds\limsup_{r\rt+\infty} \left(\frac{(1+(\log(\log r))^{\epsilon-1})\log(\log \gamma r)}{\log r}+\right.  \\
            & &\ds \hfill +\left.\frac{\log\nu_T(\gamma r)}{\log r}\right)\\
            & \leq & \varrho.
          \end{array}$$
    \end{proof}

    \begin{defn}
        Une fonction d\'erivable $\rho:\ ]0,+\infty[\rt ]0,+\infty[$ est appel\'ee \textit{ordre pr\'ecis\'e} si elle admet une limite finie $\varrho$ \`a l'infini et $$\lim_{r\to+\infty}\rho'(r).r\log r=0.$$
        Un courant positif pluriharmonique $T$ d'ordre $\varrho$ fini est dit de \textit{type minimal }(resp. \textit{normal, maximal}) par rapport \`a un ordre pr\'ecis\'e $\rho(r)$, o\`u $\rho(r)\rt \varrho$, si la limite $$\sigma(T):=\limsup_{r\to+\infty}{\nu_T(r)\over r^{\rho(r)}}=0\ (\hbox{resp. } \sigma(T)\in{}]0,+\infty[,\ \sigma(T)=+\infty).$$
    \end{defn}

    Comme cons\'equence des r\'esultats pr\'ec\'edents on g\'en\'eralise, au courant positif pluriharmonique sur $G_{q,n}$, un r\'esultat d\'emontr\'e par \textsc{Gruman} \cite{Gr} pour les ensembles analytiques et par \textsc{Amamou-Ben Farah} \cite{Am-BF} pour les courants positifs ferm\'es sur $\P^{n-1}=G_{1,n}$.

    \begin{cor} \label{cor4}
        Soit $T$ un courant positif pluriharmonique de bidimension $(p,p)$ sur $\C^n$ d'ordre $\varrho$ fini et de type normal par rapport \`a un ordre pr\'ecis\'e $\rho(r)$. Alors pour presque tout $L\in G_{q,n}$, $T_{|L}$ est de type normal ou maximal par rapport \`a $\rho(r)$.
    \end{cor}

    \begin{proof}
        Pour alleger l'\'ecriture on note par $\sigma:=\sigma(T)$ et $\sigma_L:=\sigma(T_{|L})$. Par hypoth\`ese, il existe une suite $(r_m)_m$ croissante vers $+\infty$ telle que $\nu_T(r_m).r^{-\rho(r_m)}\rt \sigma$. Soit
        $$E=\left\{L\in \Omega_T;\ \lim_{m\to +\infty}\frac{\nu_{T_{|L}}(\alpha r_m)}{\nu_T(r_m)}=0,\ \hbox{pour tout }\alpha>0\right\}\cup \Omega_T^\complement.$$ D'apr\`es le corollaire \ref{cor2}, $E$ est de mesure nulle. Pour $L\not\in E$, quitte \`a extraire une sous suite de $(r_m)_m$, il existe $\alpha_0>0$ tel que\\ $\mathfrak{a}_L:=\limsup_{m\to +\infty}\nu_{T_{|L}}(\alpha_0 r_m)/ \nu_T(r_m)\in{}]0,+\infty]$.\\
        \begin{equation}\label{eq 2.2}
            \underbrace{\nu_{T_{|L}}(\alpha_0 r_m)\over (\alpha_0 r_m)^{\rho(\alpha_0 r_m)}}_{Q_1(m)}=\underbrace{\nu_{T_{|L}}(\alpha_0 r_m)\over \nu_T(r_m)}_{Q_2(m)}\;\underbrace{\nu_T(r_m)\over r_m^{\rho(r_m)}}_{Q_3(m)}\;\underbrace{r_m^{\rho(r_m)-\rho(\alpha_0r_m)}}_{Q_4(m)}\; \underbrace{1\over \alpha_0^{\rho(\alpha_0r_m)}}_{Q_5(m)}
        \end{equation}
        Si $m\to+\infty$, la premi\`ere quantit\'e $Q_1(m)$ admet une limite-sup\'erieure inf\'erieure ou \'egale \`a $\sigma_L$, $Q_2(m)$ tend vers $\mathfrak{a}_L$, $Q_3(m)$ tend vers $\sigma$ et $Q_5(m)$ tend vers $\alpha_0^{-\varrho}$. Pour $Q_4(m)$, le cas $\alpha_0=1$ est \'evident, dans l'autre cas, d'apr\`es le th\'eor\`eme des accroissements finis, il existe $c_m$ entre $r_m$ et $\alpha_0r_m$ tel que $$\log(Q_4(m))=(\rho(r_m)-\rho(\alpha_0r_m))\log r_m= \rho'(c_m)(1-\alpha_0)r_m\log r_m$$
        donc pour $m$ suffisamment grand il existe une constante $Cte\geq 0$ tel que $$|\log(Q_4(m))|\leq Cte.|\rho'(c_m)c_m\log c_m|$$ et par suite $\log(Q_4(m))\rt 0$ si $m\to +\infty$. En effet:\\
            $\bullet$ Si $\alpha_0> 1$, on a $r_m< c_m$ donc $r_m\log r_m\leq c_m\log c_m$ et on peut prendre $Cte =\alpha_0-1$.\\
            $\bullet$ Si $\alpha_0<1$, on a $\alpha_0r_m< c_m<r_m$ donc $$r_m\log r_m=c_m\log c_m.{r_m\log r_m\over c_m\log c_m}\leq c_m\log c_m {\log r_m\over \alpha_0\log(\alpha_0r_m)}$$ et, pour $m$ suffisamment large, on peut choisir $Cte=2(1-\alpha_0)/\alpha_0$.\\
        Si on tend $m\to+\infty$ l'\'equation (\ref{eq 2.2}) donne $\sigma_L\geq \mathfrak{a}_L \sigma/\alpha_0^\varrho$ et donc $T_{|L}$ est au moins de type  normal par rapport \`a $\rho(r)$.
    \end{proof}

    \begin{rems}$ $
        \begin{itemize}
          \item S'il existe un Bor\'elien $E$ de mesure non nulle de $G_{q,n}$ et  une constante $b>0$ telle que pour presque tout $L\in E$ on ait $T_{|L}$ est de type $\sigma_L\leq b$ par rapport \`a un ordre pr\'ecis\'e $\rho(r)$ alors $T$ est de type fini par rapport \`a l'ordre pr\'ecis\'e $\rho(r)$.\\
              En effet, d'apr\`es le th\'eor\`eme \ref{th1}, il existe $c_1,\ c_2>0$ tels que pour $r>0$ (suffisamment grand) on ait
              $$\begin{array}{lcl}
                \ds\frac{\nu_T(r)}{ r^{\rho(r)}}& \leq &\ds c_1\int_E \frac{\nu_{T_{|L}}(c_2r)}{(c_2r)^{\rho(c_2r)}}\ r^{\rho(c_2r)-\rho(r)}c_2^{\rho(c_2r)}d\mu(L) \\
                &\leq & c_1b\mu(E) r^{\rho(c_2r)-\rho(r)}c_2^{\rho(c_2r)}
                \end{array}$$
              De plus, d'apr\`es la d\'emonstration du corollaire \ref{cor4}, le terme de droite de cette in\'egalit\'e admet une limite finie ($=c_1b\mu(E)c_2^\varrho$) quand $r\to+\infty$.
          \item On a prouv\'e que $$\ds\left\{L\in \Omega_T;\ \lim_{m\to +\infty}\frac{\nu_{T_{|L}}(\alpha r_m)}{\nu_T(r_m)}=+\infty\ \forall\alpha>0\right\}\subset\{L\in\Omega_T;\ \sigma_L=+\infty\}$$
              et on a le premier ensemble est de mesure nulle (voir le lemme \ref{lem3} suivant). A-t-on l'ensemble $\{L\in G_{q,n}\mo E_0;\ \sigma_L=+\infty\}$ est aussi de mesure nulle? c'est \`a dire que $T_{|L}$ est-il aussi de type normal par rapport \`a $\rho(r)$, $\mu-$presque partout?
        \end{itemize}
    \end{rems}
    \begin{lem}\label{lem3}
        Avec les m\^emes notations du corollaire \ref{cor4}, l'ensemble
        $$\mathscr A_\infty:=\left\{L\in \Omega_T;\ \exists\; \alpha_L>0,\ \lim_{m\to +\infty}\frac{\nu_{T_{|L}}(\alpha_L r_m)}{\nu_T(r_m)}= +\infty \right\}$$ est de mesure nulle.
    \end{lem}
    \begin{proof}
            Pour tout $s\in\N^*$, on consid\`ere l'ensemble $$A_s:=\left\{ L\in \Omega_T;\ \lim_{m\to +\infty}\frac{\nu_T(r_m)}{ \nu_{T_{|L}}(s r_m) }=0\right\}.$$ Si $L\in \mathscr A_\infty$ alors pour tout $s\geq \alpha_L$, on a $L\in A_s$. Donc $\mathscr A_\infty\subset \cup_{s\in\N^*}A_s$. Par suite pour montrer que $\mu(\mathscr A_\infty)=0$ il suffit de montrer que $\mu(A_s)=0$ pour tout $s\in\N^*$.\\
            Supposons qu'il existe $s_0>0$ tel que $A_{s_0}$ soit de mesure $\mu(A_{s_0})>0$. On a pour tout $L\in A_{s_0}$, $\lim_{m\to +\infty} \nu_T(r_m)/\nu_{T_{|L}}(s_0 r_m)=0$. D'apr\`es le th\'eor\`eme d'\textsc{Egorov}, on peut supposer que la convergence de cette suite vers 0 est uniforme sur $A_{s_0}$. Soit $\epsilon>0$, il existe $m_\epsilon>0$ tel que pour tout $m\geq m_\epsilon$ on a
            \begin{equation}\label{eq 2.3}
                {\nu_T(r_m)\over \nu_{T_{|L}}(s_0 r_m)}\leq \epsilon\quad \forall\;L\in A_{s_0}.
            \end{equation}
            D'autre part, d'apr\`es le lemme \ref{lem1}, pour tout $m>0$ on a  $$\int_{A_{s_0}}\nu_{T_{|L}}(s_0 r_m)d\mu(L)\leq \int_{G_{q,n}}\nu_{T_{|L}}(s_0 r_m)d\mu(L)=\nu_T(s_0 r_m).$$ Donc
            \begin{equation}\label{eq 2.4}
                \ds\int_{A_{s_0}}\frac{\nu_{T_{|L}}(s_0 r_m)}{\nu_T(s_0 r_m)}d\mu(L)\leq 1
            \end{equation}
            Les deux in\'egalit\'es (\ref{eq 2.3}) et (\ref{eq 2.4}) donnent, pour tout $m\geq m_\epsilon$,
            $$\begin{array}{lcl}
                \ds\mu(A_{s_0})\frac{\nu_T(r_m)}{ \nu_T(s_0 r_m)}&=&\ds\int_{A_{s_0}}\frac{\nu_T(r_m)}{\nu_T(s_0 r_m)}d\mu(L)\\
                & = &\ds\int_{A_{s_0}}\frac{\nu_T(r_m)}{\nu_{T_{|L}}(s_0 r_m)}\frac{\nu_{T_{|L}}(s_0 r_m)}{\nu_T(s_0 r_m)}d\mu(L)\\
                & \leq &\ds \epsilon\int_{A_{s_0}}\frac{\nu_{T_{|L}}(s_0 r_m)}{\nu_T(s_0 r_m)}d\mu(L)\leq\epsilon
              \end{array}$$  c'est \`a dire que $\lim_{m\to +\infty}\nu_T(r_m)/ \nu_T(s_0 r_m)=0$.
            Donc $$\begin{array}{lcl}
                     \sigma & := &\ds \lim_{m\to +\infty}\frac{\nu_T(r_m)}{r_m^{\rho(r_m)}} \\
                      & = &\ds \lim_{m\to+\infty}\left(\frac{\nu_T(r_m)}{\nu_T(s_0 r_m)}\times \frac{\nu_T(s_0 r_m)}{(s_0 r_m)^{\rho(s_0 r_m)}}\times s_0^{\rho(s_0 r_m)}\times r_m^{\rho(s_0 r_m)-\rho( r_m)}\right)\\
                      & = & 0
                   \end{array}$$ ce qui est absurde car $\sigma > 0$.
        \end{proof}
        Modulo une petite perturbation, il existe un ordre pr\'ecis\'e $\chi(r)$ par rapport auquel $T$ ainsi que $T_{|L}$, pour presque tout $L\in G_{q,n}$, sont de types minimales et c'est donn\'e par le lemme suivant:
        \begin{lem}\label{lem4}
            Si $T$ est de type normal par rapport \`a un ordre pr\'ecis\'e $\rho(r)$ alors pour presque tout $L\in G_{q,n}$, $T_{|L}$ est de type fini $($minimal ou normal$)$ par rapport \`a l'ordre pr\'ecis\'e
            $$\chi(r)=\left\{\begin{array}{ll}
                               \ds\rho(r)+\frac{\log(\log(e-1+r))}{\log r} & si\ r\neq1 \\
                               \rho(1)+\frac 1e & si\  r=1
                             \end{array}\right.$$
            En particulier, si $T$ est de type minimal par rapport \`a $\rho(r)$ alors $T$ et $T_{|L}$, pour presque tout $L\in G_{q,n}$, sont de types minimales par rapport \`a $\chi(r)$.
        \end{lem}

    \begin{proof}
        V\'erifiant d'abord que $\chi(r)$ est bien un ordre pr\'ecis\'e. En effet $\chi$ est d\'erivable sur $]0+\infty[$ et on a $\chi(r)\underset{r\to+\infty}\rt\varrho$ et
        $$\chi'(r)=\rho'(r)+\frac1{(e-1+r)\log r}\frac1{\log(e-1+r)}-\frac1{r\log r}\frac{\log(\log(e-1+r))}{\log(r)}$$ donc $r\log(r)\chi'(r)\underset{r\to+\infty}\rt 0$.\\
        D'apr\`es l'in\'egalit\'e (\ref{eq 2.1}), pour tout $s\in\N,\ s\geq 2$, en prenant $\epsilon=1/s$ et $\gamma=1+1/s$, il existe un ensemble n\'egligeable ${\mathscr F}_s$ de $G_{q,n}$ v\'erifiant pour tout $L\not\in {\mathscr F}_s$,
        $$\nu_{T_{|L}}(r)\leq \nu_T((1+\frac1 s)r)\left(\log(1+\frac1 s) r\right)^{1+(\log(\log r))^{-1+1/s}}.$$
        Soit $\mathscr F:=\cup_{s\geq 2}\mathscr F_s$. Alors $\mathscr F$ est n\'egligeable et pour tout $L\not\in {\mathscr F}$ on a $$\nu_{T_{|L}}(r)\leq \nu_T((1+\frac1 s)r)\left(\log(1+\frac1 s)r\right)^{1+(\log(\log r))^{-1+1/s}}\quad\forall\; s\geq 2.$$
        Comme la fonction  $\nu_T$ est semi-continue sup\'erieurement, si on fait tendre $s$ vers $+\infty$ dans l'in\'egalit\'e pr\'ec\'edente on obtient
        $$\nu_{T_{|L}}(r)\leq \left(\log r\right)^{1+1/\log(\log r)}\nu_T(r)=e \log(r)\nu_T(r).$$
        Par suite
        $$\limsup_{r\to+\infty}\frac{\nu_{T_{|L}}(r)}{r^{\chi(r)}}\leq \limsup_{r\to+\infty}e \log(r)\frac{\nu_T(r)}{r^{\rho(r)}\log(e-1+r)}= e\sigma(T).$$
    \end{proof}

\section{Ordres directionnels}
    Dans cette section on s'int\'eresse au ordres (et ordres directionnels) des courants positifs  de bidegr\'e $(k,k)$ dans $\C^N=\C^n\times \C^m$ o\`u $k\leq n$; on utilise alors les notations  $\beta_z=dd^c|z|^2$, $\beta_t=dd^c|t|^2$, $\alpha_z=dd^c\log|z|^2$ et $\alpha_t=dd^c\log|t|^2$ pour tout $(z,t)\in \C^n\times \C^m$.\\
    On a besoin du lemme de type \textsc{Lelong-Jensen} suivant:
    \begin{lem}\label{lem5} \cite{Fe}
        Soient $T$ un courant positif plurisousharmonique  de bidegr\'e $(k,k)$ dans $\C^N=\C^n\times \C^m$ o\`u $k<n$ et $D$ un bor\'elien relativement compact de $\C^m$. Alors pour tous $0<r_1<r_2$,
        $$\begin{array}{lcl}
            A(r_1,r_2)&:=&\ds \frac1{r_2^{2(n-k)}}\int_{\mathbb{B}_n(r_2)\times D}T\w \beta_z^{n-k}\w\beta_t^m+\\
            & & \hfill- \ds\frac1{ r_1^{2(n-k)}} \int_{\mathbb{B}_n(r_1)\times D}T\w \beta_z^{n-k}\w\beta_t^m\\
            & = &\ds\int_{\mathbb{B}_n(r_1,r_2)\times D}T\w \alpha_z^{n-k}\w\beta_t^m +\\
            & &\hfill+\ds\int_{r_1}^{r_2}\left(\frac1{ s^{2(n-k)}}-\frac1{r_2^{2(n-k)}}\right)sds\int_{\mathbb{B}_n(s)\times D}dd^cT\w \beta_z^{n-k-1}\w\beta_t^m\\
            & &\ds \hfill+ \left(\frac1{r_1^{2(n-k)}}-\frac1{r_2^{2(n-k)}}\right) \int_0^{r_1}sds\int_{\mathbb{B}_n(s)\times D}dd^cT\w \beta_z^{n-k-1}\w\beta_t^m.
        \end{array}$$
    \end{lem}
    Par cons\'equent, la fonction $$\ds r\longmapsto \mathscr N_{(T,D)}(r):=\frac{1}{r^{2(n-k)}}\int_{B_n(r)\times D}T\w \beta_z^{n-k}\w\beta_t^m$$ est positive et croissante, ce qui explique l'existence du nombre de \textsc{Lelong} directionnel, $\mathscr N_{(T,D)}(0):=\lim_{r\to0^+}\mathscr N_{(T,D)}(r)$, de $T$ en 0 par rapport \`a $D$ suivant la direction de $\C^n$. On d\'efinit alors l'\textit{ordre directionnel} de $T$ par rapport \`a $D$, suivant la direction de $\C^n$, par:
    $$\varrho_{(T,D)}=\limsup_{r\to+\infty}\frac{\log \mathscr N_{(T,D)}(r)}{\log r}.$$
    
    \begin{prop}\label{prop1}
        Soient $T$ un courant positif plurisousharmonique  de bidegr\'e $(k,k)$ sur $\C^N=\C^n\times \C^m$ o\`u $k<n$ et $D$ un bor\'elien relativement compact de $\C^m$. Si $T$ est d'ordre $\varrho$ fini  alors $T$ est d'ordre directionnel $\varrho_{(T,D)}$ fini qui v\'erifie $\varrho_{(T,D)}\leq 2m+\varrho$.
    \end{prop}
    \begin{proof}
        Si $\delta>1$ alors il existe $r$ suffisamment grand tel que $\mathbb{B}_n(r)\times D\subset \mathbb{B}_N(\delta r)$.
        $$\begin{array}{lcl}
            \ds(\delta r)^{2m}\nu_T(\delta r) & = &\ds\frac1{(\delta r)^{2(n-k)}}\int_{\mathbb{B}_N(\delta r)}T\w(\beta_z+\beta_t)^{N-k} \\
            & \geq &\ds \frac1{(\delta r)^{2(n-k)}}\int_{\mathbb{B}_n(r)\times D}T\w\beta_z^{n-k}\w\beta_t^m\\
            & \geq& \ds\delta^{-2(n-k)}\mathscr N_{(T,D)}(r).
        \end{array}$$
        Donc $\ds\frac{\log \left[(\delta r)^{2m}\nu_T(\delta r)\right]}{\log r}\geq \frac{\log (\delta^{-2(n-k)}\mathscr N_{(T,D)}(r))}{\log r}$ et en passant \`a la limite sup\'erieure quand $r$ tend vers $+\infty$, on obtient  $2m+\varrho\geq \varrho_{(T,D)}$ et la proposition est prouv\'ee.
    \end{proof}

    Dans la suite on s'int\'eresse \`a la r\'eciproque, \`a savoir la question suivante~: si $T$ est d'ordre directionnel fini, a-t-on que $T$ est d'ordre fini?
    Une r\'eponse partielle positive est donn\'ee  par le th\'eor\`eme \ref{th3}; pour le citer on a besoin de quelques notions: Pour $B$ un bor\'elien relativement compact de $\C^n$, on d\'efinit de la m\^eme mani\`ere l'ordre directionnel de $T$ par rapport \`a $B$ suivant la direction de $\C^m$ comme \'etant $\varrho_{(B,T)}:=\limsup_{r\to+\infty}\frac{\log \mathscr M_{(B,T)}(r)}{\log r}$ o\`u $\mathscr M_{(B,T)}(r)=\frac1{r^{2(m-k)}}\int_{B\times \mathbb{B}_m(r)}T\w \beta_z^n\w\beta_t^{m-k}.$ Le lemme suivant sera utile pour la suite, et sa d\'emonstration  est analogue \`a celle du lemme \ref{lem2}.
    \begin{lem}\label{lem6}
        Soient $S$ un courant positif de $dd^c-$n\'egatif de bidegr\'e $(k,k)$  sur $\C^N=\C^n\times \C^m$ o\`u $k<n$ et $D$ un bor\'elien relativement compact de $\C^m$. Soit $f$ une fonction psh, $f\geq -1$,  de classe $\mathcal C^2$ sur un ouvert $\mathcal O$ de $\C^n$ telle que $\mathcal O'=\{z\in \mathcal O;\ f(z)<0\}$ soit relativement compact dans $\mathcal O$. Soit $K$ un compact de $\mathcal O'$, on pose $c_K=-\sup_{z\in K}f(z)$.\\
        Alors pour tout entier $1\leq s\leq n-k$ et pour toute fonction $g$ psh de classe $\mathcal C^2$ sur $\mathcal O'$ v\'erifiant $-1\leq g<0$ on a:
        $$\int_{K\times D} S\w(dd^cg)^{n-k}\w\beta_t^m\leq c_K^{-s}\int_{\mathcal O'\times D}S\w(dd^cf)^s\w(dd^cg)^{p-s}\w\beta_t^m.$$
    \end{lem}

    \begin{theo}\label{th3}
        Soit $T$ un courant positif pluriharmonique  de bidegr\'e $(1,1)$ sur $\C^N=\C^n\times \C^m$ o\`u $n,\ m>1$. On suppose qu'il existe deux compacts $D$ et $D'$ d'int\'erieurs non vide  de $\C^n$ et $\C^m$ respectivement tels que $T$ soit d'ordres finis dans les directions de $\C^m$ et $\C^n$  par rapport \`a $D$ et $D'$, alors $T$ est d'ordre fini.
    \end{theo}

    \begin{proof}
        Pour $r>0$ on a $\mathbb{B}_N(r)\subset K_r:=\mathbb{B}_n(r)\times \mathbb{B}_m(r)$, donc
        \begin{equation}\label{eq 3.2}
            \begin{array}{lcl}
                \nu_T(r) & = &\ds \frac1{r^{2(n+m-1)}}\int_{\mathbb{B}_N(r)}T\w(\beta_z+\beta_t)^{n+m-1} \\
                & \leq & \ds \frac1{r^{2(n+m-1)}}\int_{K_r} T\w(\beta_z+\beta_t)^{n+m-1}
            \end{array}
        \end{equation}
        Comme $D$ est un compact d'int\'erieur non vide de $\C^n$, il existe une fonction psh $u$ (la fonction extr\'emale de \textsc{Siciak} associ\'e \`a $D$ ) de classe $\mathcal C^2$ sur $\C^n$ telle que la mesure $(dd^cu)^n$ soit \`a support dans $D$. De m\^eme il existe une fonction $v$ psh de classe $\mathcal C^2$ sur $\C^m$ telle que la mesure $(dd^cv)^m$ soit port\'ee par $D'$. De plus elles v\'erifient $\max(\log|z|,-1)\leq u(z)\leq \log|z|+A$ pour tout $z\in\C^n$ et $\max(\log|t|,-1)\leq v(t)\leq \log|t|+C$ pour tout $t\in\C^m$  o\`u $A$ et $C$ sont deux constantes. Consid\'erons $\epsilon>0$ et la fonction
        $$w(z,t)=\left(\frac{u(z)-A}{\log(1+2r)}-1\right)+\left(\frac{v(t)-C}{\log(1+2r)}-1\right). $$ Alors $w$ est une fonction psh de classe $\mathcal C^2$ sur $\C^N$ qui v\'erifie, pour $r$ assez grand, $w>-\epsilon$ sur le bord de $K_{2r}$. Donc l'ensemble ${\mathcal O}:=\{(z,t)\in\C^N;\ w(z,t)+2\epsilon<0\}$ est relativement compact dans  $K_{2r}$. De plus $K_r\subset\subset{\mathcal O}$. Soient $$c_r=-\sup_{(z,t)\in \overline{K_r}}w(z,t)=O\left(\frac1{\log(1+2r)}\right)\hbox{ et }g(z,t) =\frac{|z|^2+|t|^2-4r^2}{4r^2}.$$ D'apr\`es le lemme \ref{lem2},
        $$\begin{array}{lcl}
            \ds\int_{K_r} T\w(dd^cg)^{n+m-1}& \leq &\ds c_r^{-(n+m-1)}\int_{\mathcal O}T\w(dd^cw)^{n+m-1} \\
            & \leq &\ds c_r^{-(n+m-1)}\int_{K_{2r}}T\w(dd^cw)^{n+m-1}
        \end{array}$$
        Donc
        \begin{equation}\label{eq 3.3}
            \frac1{r^{2(n+m-1)}}\int_{K_r}T\w(\beta_z+\beta_t)^{n+m-1}\leq \theta(r)\Theta_T(2r) .
        \end{equation}
        o\`u $\ds\theta(r)=\left(\frac4{c_r\log(1+2r)}\right)^{n+m-1}$ et $\ds\Theta_T(2r)= \int_{K_{2r}} T\w(dd^c(u+v))^{n+m-1}$. Comme $c_r\log(1+2r)$ est born\'ee ind\'ependamment de $r$, il existe $a>0$ tel que
        $\theta(r)\leq a$ pour tout $r>0$ (suffisamment grand). Par raison de degr\'e,

        \begin{equation}\label{eq 3.4}
            \begin{array}{lcl}
                \Theta_T(2r)& =& \ds\int_{K_{2r}}T\w(dd^c(u+v))^{n+m-1}\\
                & = &\ds C_{N-1}^m\int_{K_{2r}}T\w(dd^cu)^{n-1}\w(dd^cv)^m+C_{N-1}^n \int_{K_{2r}} T\w(dd^cu)^n\w\\
                & & \hfill\w(dd^cv)^{m-1}\\
                &= &\ds C_{N-1}^m\int_{\mathbb{B}_n(2r)\times D'}T\w(dd^cu)^{n-1}\w(dd^cv)^m+\\
                & &\ds \hfill+ C_{N-1}^n \int_{D\times \mathbb{B}_m(2r)}T\w(dd^cu)^n\w(dd^cv)^{m-1}\\
                &\leq &\ds b_1C_{N-1}^m\int_{\mathbb{B}_n(2r)\times D'}T\w(dd^cu)^{n-1}\w\beta_t^m+\\
                & & \ds \hfill+b_2 C_{N-1}^n \int_{D\times \mathbb{B}_m(2r)}T\w\beta_z^n\w(dd^cv)^{m-1}
            \end{array}
        \end{equation}
        o\`u $b_1$ et $b_2$ sont deux constantes positives qui d\'ependent uniquement de $D'$ et $D$ respectivement.\\
        Soient $f(z)=\frac{|z|^2-(3r)^2}{(3r)^2}$ et $g_2(z)=\frac{u(z)-A-\log(3r)}{\log(\kappa r)}$ o\`u $\kappa:=1+3e^{1+A}$, $f$ est psh sur $\C^n$, $-1<f(z)<0$ sur $\mathbb{B}_n(3r)$ et $-1\leq g_2(z)<0$ sur $\mathbb{B}_n(3r)$. Le lemme \ref{lem6} appliqu\'e \`a $f,\ g$ et $K=\overline{\mathbb B}_n(2r)$ donne

        $$\ds\int_{ \mathbb{B}_n(2r)\times D'}T\w(dd^cg_2)^{n-1}\w\beta_t^m\leq \ds \left(\frac95\right)^{n-1}\int_{ \mathbb{B}_n(3r)\times D'}T\w(dd^cf)^{n-1}\w\beta_t^m $$
        on obtient donc
        $$\int_{ \mathbb{B}_n(2r)\times D'}T\w(dd^cu)^{n-1}\w\beta_t^m\leq \left(\frac95\log(\kappa r)\right)^{n-1}\mathscr N_{(T,D')}(3r). $$
        De la m\^eme fa\c{c}on on d\'emontre que
        $$\int_{D\times \mathbb{B}_m(2r)}T\w\beta_z^n\w(dd^cv)^{m-1}\leq \left(\frac95\log(\kappa r)\right)^{m-1}\mathscr M_{(D,T)}(3r).$$
        L'in\'egalit\'e (\ref{eq 3.4}) donne
        \begin{equation}\label{eq 3.5}
            \begin{array}{lcl}
                \ds\int_{K_{2r}}T\w(dd^c(u+v))^{n+m-1} & \leq &\ds b_1C_{N-1}^m  \left(\frac95\log(\kappa r)\right)^{m-1}\mathscr M_{(D,T)}(3r)+\\
                & & \ds+ b_2C_{N-1}^n \left(\frac95\log(\kappa r)\right)^{n-1}\mathscr N_{(T,D')}(3r).
            \end{array}
        \end{equation}
        D'apr\`es les in\'egalit\'es (\ref{eq 3.2}), (\ref{eq 3.3}) et (\ref{eq 3.5}), on d\'eduit que
        $$\nu_T(r)\leq c_1(\log(\kappa r))^{n-1}\mathscr N_{(T,D')}(3r)+c_2(\log(\kappa r))^{m-1}\mathscr M_{(D,T)}(3r)$$
        o\`u $c_1, c_2$ sont deux constantes positives. Un calcul simple montre alors que
        $$\frac{\log(\nu_T(r))}{\log r}\leq \frac{\tau+N\log\log(\kappa r)}{\log r}+\max\left(\frac{\log(\mathscr N_{(T,D)}(3r))}{\log r},\frac{\log(\mathscr M_{(D',T)}(3r))}{\log r}\right)$$
        et par  passage \`a la limite sup\'erieure quand $r$ tend vers $+\infty$, on trouve
        $\varrho\leq \max(\varrho_{(T,D)},\varrho_{(D',T)})$.
    \end{proof}
\section*{Remerciements}
    Nous remercions vivement les professeurs \textsc{Jean-Pierre Demailly, Hassine El Mir} et \textsc{H\`edi Ben Messaoud} pour d'utiles conversations \`a propos ce travail.

\end{document}